\documentclass{amsart}

\usepackage{amssymb, amsmath}
\usepackage{mathrsfs}
\usepackage{amscd}
\usepackage{verbatim}
\usepackage{enumerate}
\usepackage{a4wide}

\theoremstyle{plain}
\newtheorem{theorem}{Theorem}[section]
\theoremstyle{remark}
\newtheorem{remark}[theorem]{Remark}
\newtheorem{facts}[theorem]{Facts}

\theoremstyle{plain}
\newtheorem{corollary}[theorem]{Corollary}
\newtheorem{lemma}[theorem]{Lemma}
\newtheorem{proposition}[theorem]{Proposition}
\newtheorem{definition}[theorem]{Definition}

\numberwithin{equation}{section}

\def\R{{\mathbb R}}

\newcommand{\E}{{\mathbb E}}
\renewcommand{\P}{{\mathbb P}}

\newcommand{\F}{{\mathcal F}}

\newcommand{\g}{\gamma}

\newcommand{\eps}{\varepsilon}

\renewcommand{\O}{\Omega}

\newcommand{\calL}{{\mathcal L}}
\newcommand{\n}{\Vert}

\newcommand{\lb}{\langle}
\newcommand{\rb}{\rangle}

\begin{document}

\author{Stanislaw Kwapie{\'n}}
\address{Institute of Mathematics
\\ Warsaw University
\\ Banacha 2
\\ 02-097 Warszawa
\\ Poland
}
\email{KwapStan@mimuw.edu.pl}

\author{Mark Veraar}
\address{Delft Institute of Applied Mathematics\\
Delft University of Technology \\ P.O. Box 5031\\ 2600 GA Delft\\The
Netherlands} \email{M.C.Veraar@tudelft.nl}

\author{Lutz Weis}
\address{Institut f\"ur Analysis \\
Karlsruhe Institute of Technology\\
D-76128  Karls\-ruhe\\Germany}
\email{lutz.weis@kit.edu}

\thanks{The first named author is supported by NCN grant Dec-2012/05/B/ST1/00412. The second author
is supported by the VIDI subsidy 639.032.427 of the Netherlands Organisation for Scientific Research (NWO).
The third named author is supported by a grant from the Graduierten Kolleg 1294DFG.}

\date\today

\title
[$R$-boundedness versus $\g$-boundedness]{$R$-boundedness versus $\g$-boundedness}

\begin{abstract}
It is well-known that in Banach spaces with finite cotype, the $R$-bounded and $\g$-bounded families of operators coincide. If in addition $X$ is a Banach lattice, then these notions can be expressed as square function estimates. It is also clear that $R$-boundedness implies $\g$-boundedness. In this note we show that all other possible inclusions fail. Furthermore, we will prove that $R$-boundedness is stable under taking adjoints if and only if the underlying space is $K$-convex.
\end{abstract}



\maketitle

\section{Introduction}
Square function estimates of the form
\begin{equation}\label{eq:MarcZyg}
\Big\|\Big(\sum_{n=1}^N |T_n x_n|^2\Big)^{1/2}\Big\|_{L^q} \leq C \Big\|\Big(\sum_{n=1}^N |x_n|^2\Big)^{1/2}\Big\|_{L^p}
\end{equation}
for operators $T_1, \ldots, T_N:L^p(\R^d)\to L^q(\R^d)$ and $x_1, \ldots, x_N\in L^p(\R^d)$ with $1<p,q<\infty$, play an important role in harmonic analysis, in particular in Calderon-Zygmund and martingale theory. In 1939 Marcinkiewicz and Zygmund \cite{MarcZyg} (building on previous work of Paley \cite{Paley}, see also \cite{GarciaRubio}) proved \eqref{eq:MarcZyg} for a single linear operator $T = T_1 = \ldots = T_N:L^p\to L^q$ by expressing the square functions in terms of random series, i.e.
\begin{equation}\label{eq:randomseries}
\Big\|\Big(\sum_{n=1}^N |x_n|^2\Big)^{1/2}\Big\|_{L^p} \eqsim_p \E\Big\|\sum_{n=1}^N r_n x_n \Big\|_{L^p} \eqsim_p \E\Big\|\sum_{n=1}^N\gamma_n x_n \Big\|_{L^p},
\end{equation}
where $(\gamma_n)_{n\geq 1}$ are independent standard Gaussian random variables and $(r_n)_{n\geq 1}$ are independent Rademacher random variables. Such random series with values in a Banach space have become a central tool in the geometry of Banach spaces and probability theory in Banach spaces (see \cite{AlKa}, \cite{KwWo}, \cite{LeTa} and \cite{MaPi}).

Random series also allow to extend \eqref{eq:MarcZyg} to general Banach spaces
and have become an effective tool to extend many central results about Fourier multipliers, Calderon--Zygmund operators, stochastic integrals and the holomorphic functional calculus to Banach space valued functions and ``integral operators'' with operator-valued kernels
(e.g.\ see \cite{ArBu}, \cite{Bou2}, \cite{CPSW}, \cite{DHP}, \cite{HytMih}, \cite{JMX}, \cite{KaWe}, \cite{KuWe}, \cite{NVW11eq} and \cite{Weis-maxreg}). In recent years it was observed that many of the classical results extend to the operator-valued setting as long as all uniform boundedness assumptions are replaced by $R$-boundedness or $\gamma$-boundedness assumptions (see the next section for the precise definition).
In many of these results it is crucial that the Banach space $X$ has finite cotype and in this case the second part of \eqref{eq:randomseries} remains valid: (see \cite[Lemma 4.5 and Proposition 9.14]{LeTa})
\begin{equation*}
\E\Big\|\sum_{n=1}^N r_n x_n \Big\|_{X} \eqsim_X \E\Big\|\sum_{n=1}^N\gamma_n x_n \Big\|_{X}.
\end{equation*}
For this reason $R$-boundedness and $\g$-boundedness are equivalent under finite cotype assumptions. Furthermore, it is well-known that $R$-boundedness always implies $\g$-boundedness. It was an open problem whether these two notions are the same for all Banach spaces.

By constructing an example in $\ell^\infty_n$'s and combining this with methods from the geometry of Banach spaces we prove the following result:
\begin{theorem}\label{thm:main}
Let $X$ and $Y$ be nonzero Banach spaces. The following assertions are equivalent:
\begin{enumerate}[(i)]
\item Every $\g$-bounded family $\mathscr{T}\subseteq \calL(X,Y)$ is $R$-bounded.
\item $X$ has finite cotype.
\end{enumerate}
In this case $\mathcal{R}(\mathscr{T}) \lesssim_X \mathcal{R}^\gamma(\mathscr{T}) \leq \mathcal{R}(\mathscr{T})$.
\end{theorem}
In Section \ref{sec:square} we will also discuss the connections between $R$-boundedness and $\g$-boundedness and
$\ell^2$-boundedness (as defined in \eqref{eq:MarcZyg} and Section \ref{sec:square}) for general lattices.
We show that $\ell^2$-boundedness implies $R$-boundedness if and only if the codomain $Y$ has finite cotype. Furthermore, $R$-boundedness implies $\ell^2$-boundedness if and only if the domain $X$ has finite cotype.
The proofs are based on connections with classical notions such as $p$-summing operators and operators of cotype $q$.
These connections and the deep result of Montgomery-Smith and Talagrand, on cotype of operators from $C(K)$,
(which are summarized in Talagrand's recent monograph \cite{Talagrnew}, chapter
16) allow to obtain as quick consequences proofs of Theorem \ref{thm:main} and
Theorem \ref{thm:Lmain1}.
Since the results of Montgomery-Smith and Talagrand are quite involved
and we need for the proof of Theorem \ref{thm:main} a simple case we decided to
give in Section 3 an elementary and a concise proof of Theorem 1.1
which did not refer to the results on the cotype of operators.
However we have to underline that the ideas behind this proof are the
same as in the proof of \cite[Theorem 5.3, page 33]{Montgomthesis}.

In Section \ref{sec:dual} we will characterize when $R$-boundedness and $\gamma$-boundedness are stable under taking adjoints. It is well-known that the notion of $K$-convexity is a sufficient condition for this. We will prove that it is also necessary. Surprisingly the proof of this result is based on similar techniques as in Section \ref{sec:square}.

\medskip
\noindent \textbf{Acknowledgement:} The authors thank the anonymous referee for helpful comments.

\section{Preliminaries}

Let $(r_n)_{n\geq 1}$ be a Rademacher sequence defined on a probability space $(\O_r, \F_r, \P_r)$, i.e. $\P(r_1 = 1) = \P(r_1 = -1) =1/2$ and $(r_n)_{n\geq 1}$ are independent and identically distributed. Let $(\g_n)_{n\geq 1}$ be a Gaussian sequence defined on a probability space $(\O_\g, \F_\g, \P_\g)$, i.e. $(\g_n)_{n\geq 1}$ are independent standard Gaussian random variables. Expectation with respect to the Rademacher sequence and Gaussian sequence are denoted by $\E_r$ and $\E_\g$ respectively. The expectation on the product space will be denoted by $\E$.

For Banach spaces $X$ and $Y$, the bounded linear operators from $X$ to $Y$ will be denoted by $\calL(X, Y)$.

\begin{definition}
Let $X$ and $Y$ be Banach spaces. Let $\mathscr{T}\subseteq \calL(X,Y)$
\begin{enumerate}[(i)]
\item The set of operators $\mathscr{T}$ is called {\em $\g$-bounded} if there exists a constant $C\ge 0$ such that for all $N\geq 1$, for all $(x_n)_{n=1}^N$ in $X$ and
$(T_n)_{n=1}^N$ in ${\mathscr {T}}$ we have
 \begin{equation}\label{eq:rbdd}
 \Big(\E \Big\n \sum_{n=1}^N \g_n T_n x_n\Big\n^2\Big)^{1/2} \le C\Big(\E \Big\n \sum_{n=1}^N \g_n x_n\Big\n^2\Big)^{1/2}.
\end{equation}
The least admissible constant $C$ is called the {\em $\g$-bound} of
$\mathscr {T}$, notation $\mathcal{R}^\g(\mathscr{T})$.
\item If the above holds with $(\g_n)_{n\geq 1}$ replaced by $(r_n)_{n\geq 1}$, then
$\mathscr{T}$ is called {\em $R$-bounded}. The {\em $R$-bound} of $\mathscr{T}$ will be denoted by
$\mathcal{R}(\mathscr{T})$.
\item If $\mathscr{T}$ is uniformly bounded we write $\mathcal{U}(\mathscr{T}) = \sup_{T\in \mathscr{T}} \|T\|$.
\end{enumerate}
\end{definition}
We refer to \cite{CPSW} and \cite{KuWe} for a detailed discussion on $R$-boundedness. Let us note that by the Kahane-Khincthine inequalities (see \cite[Theorem 4.7]{LeTa}) the second moments may be replaced by any $p$-th moment with $p\in (0,\infty)$.

\begin{remark}\label{rem:equal}
Some of the operators $T_n$ in \eqref{eq:rbdd} could be identical. This sometimes leads to difficulties. However, for $R$-boundedness a randomization argument shows that it suffices to consider distinct operators $T_1, \ldots, T_N\in \mathscr{T}$ (see \cite[Lemma 3.3]{CPSW}). Unfortunately, such a result is not known for $\g$-boundedness.
\end{remark}

An obvious fact which we will use below is the following: Let $\mathscr{T}\subseteq \calL(X,Y)$ be $R$-bounded. If $U:E\to X$ and $V:Y\to Z$ are bounded operators, then
\begin{equation}\label{eq:obvious}
\mathcal{R}(\{VTU: T\in \mathscr{T}\} \leq \|V\| \mathcal{R}(\mathscr{T}) \|U\|.
\end{equation}
The same holds for $\g$-boundedness.

For details on type and cotype, we refer to \cite[Chapter 11]{DJT} and \cite{LeTa}. For type and cotype of operators we refer to \cite{PieWen} and \cite{Talagrnew} and references therein.

Let $q\in [2, \infty]$.
An operator $T\in \calL(X,Y)$ is said to be of {\em Rademacher cotype $q$} if there is a constant $C$ such that for all $N\geq 1$, and $x_1, \ldots, x_N\in X$ one has
\begin{equation*}
\Big(\sum_{n=1}^N \|T x_n\|^q\Big)^{1/q}\leq C \Big\|\sum_{n=1}^N r_n x_n\Big\|_{L^q(\Omega;X)}
\end{equation*}
The infimum of all constants $C$ is denoted by $C_q(T)$. Replacing $(r_n)_{n\geq 1}$ by $(\gamma_n)_{n\geq 1}$ one obtains the definition of Gaussian cotype $q$ of $T$ and the optimal constant in this case is denoted by $C_q^{\gamma}(T)$. It is well-known that this notion is different in general (see Remark \ref{rem:montgom}). In the case $X = Y$ and $T$ is the identity, one obtains the notions of Rademacher and Gaussian cotype $q$ of $X$, and these notions are known to be equivalent (see \cite{DJT} and \cite{LeTa}).

Let $p\in [1, 2]$. An operator $T\in \calL(X,Y)$ is said to be of {\em Rademacher type $p$} if there is a constant $\tau$ such that for all $N\geq 1$, and $x_1, \ldots, x_N\in X$ one has
\begin{equation*}
\Big\|\sum_{n=1}^N r_n T x_n\Big\|_{L^p(\Omega;Y)} \leq \tau \Big(\sum_{n=1}^N \|x_n\|^p\Big)^{1/p}
\end{equation*}
The infimum of all constants $\tau$ is denoted by $\tau_p(T)$. Replacing $(r_n)_{n\geq 1}$ by $(\gamma_n)_{n\geq 1}$ one obtains the definition of Gaussian type $p$ of $T$ and the optimal constant in this case is denoted by $\tau_q^{\gamma}(T)$. By an easy randomization argument and \cite[Lemma 4.5]{LeTa} these notions can be seen to be equivalent. In the case $X = Y$ and $T$ is the identity, one obtains the notions of Rademacher and Gaussian type $p$ of $X$. We say that $X$ has {\em nontrivial type} if there exists a $p\in (1, 2]$ such that $X$ has type $p$.

The Maurey--Pisier theorem \cite[Theorem 1.1]{MaPi} gives a way to check whether a given Banach space $X$ has finite cotype. In order to state this result recall that for $p\in [1, \infty]$ and $\lambda>1$, $X$ contains $\ell^p_n$'s $\lambda$-uniformly if for every $n\geq 1$, there exists a mapping $J_{n}:\ell^p_n\to X$ such that
\[\lambda^{-1}\|x\|\leq \|J_n x\|\leq \|x\|, \ \ \ x\in \ell^p_n.\]

\begin{theorem}\label{thm:MP1}
For a Banach space $X$ the following are equivalent:
\begin{enumerate}[(i)]
\item $X$ does not have finite cotype.
\item $X$ contains $\ell^\infty_n$'s $\lambda$-uniformly for some (for all) $\lambda>1$.
\end{enumerate}
\end{theorem}

There is a version for type as well:
\begin{theorem}\label{thm:MP2}
For a Banach space $X$ the following are equivalent:
\begin{enumerate}[(i)]
\item $X$ does not have nontrivial type.
\item $X$ contains $\ell^1_n$'s $\lambda$-uniformly for some (for all) $\lambda>1$.
\item $X^*$ does not have nontrivial type.
\end{enumerate}
\end{theorem}
In \cite{Pis82} it was shown that another equivalent statement is that $X$ is $K$-convex. For a detailed treatment of these results and much more, we refer to \cite[Theorem 11.1.14]{AlKa}, \cite[Chapter 13 and 14]{DJT}, \cite{Mau03} and \cite{MilSch86}.

Finally we state a simple consequence of Theorem \ref{thm:MP1} which will be applied several times.
\begin{corollary}\label{cor:MP1}
If $X$ does not have finite cotype, then for every $N\geq 1$, there exist $J_N:\ell^\infty_N\to X$ and $\widehat{I}_N:X\to \ell^\infty_N$ such that $\|J_N\|\leq 1$, $\|\widehat{I}_N\|\leq 2$
\[\widehat{I}_N J_N = \text{id}_{\ell^\infty_N} \ \ \text{and} \ \ J_N\widehat{I}_N|_{X_0} = \text{id}_{X_0},\]
where $X_0 = J_N \ell^\infty_N$.
\end{corollary}
\begin{proof}
Fix $N\geq 1$. By the Maurey-Pisier Theorem \ref{thm:MP1} we can find a bounded linear operator $J_N:\ell_N^\infty\to X$ such that $\frac{1}{2} \|x\|\leq \|J_N x\| \leq \|x\|$. Let $X_0 = J_N \ell_N^\infty$. Let $I_N:X_0\to \ell^\infty_N$ be the invertible operator given by $I_N x = e$ when $J_N e =x$. Let $(e_n^*)_{n=1}^N$ be the standard basis in $\ell^1$. For each $1\leq n\leq N$ let $x_n^* = I_N^* e_n^*\in X_0^*$ and let $z_n^*\in X^*$ be a Hahn-Banach extension of $x_n^*$. Then $\widehat{I}_N:X\to \ell^\infty_N$ given by $\widehat{I}_N x = (\lb x,z_n^*\rb)_{n=1}^N$ is an extension of $I_N$ which satisfies $\|\widehat{I}_N\| = \|I_N\|\leq 2$. From the construction it is clear that $\widehat{I}_N J_N = I_N J_N = \text{id}_{\ell^\infty_N}$.
\end{proof}

\begin{facts}
Let $X$ be a Banach space and let $p\in [1, \infty)$. The following hold:
\begin{enumerate}[(i)]
\item One always has
\begin{equation}\label{eq:randomseries2a}
\Big\|\sum_{n=1}^N r_n x_n \Big\|_{L^p(\Omega;X)} \leq \Big(\frac{\pi}{2}\Big)^{1/2}\Big\|\sum_{n=1}^N\gamma_n x_n \Big\|_{L^p(\Omega;X)},  \ \ \  x_1, \ldots, x_N\in X, \ N\geq 1.
\end{equation}
\item The space $X$ has finite cotype if and only if there is a constant $C$ such that
\begin{equation}\label{eq:randomseries2b}
\Big\|\sum_{n=1}^N \gamma_n x_n \Big\|_{L^p(\Omega;X)} \leq C \Big\|\sum_{n=1}^N r_n x_n \Big\|_{L^p(\Omega;X)},  \ \ \  x_1, \ldots, x_N\in X, \ N\geq 1.
\end{equation}
\end{enumerate}
\end{facts}
For (i) see \cite[Proposition 12.11]{DJT}. For (ii) see \cite[Proposition 12.27]{DJT} and \cite[Chapter 9]{LeTa}.

\begin{remark}\label{rem:montgom}
If $X$ has finite cotype, then it follows from \eqref{eq:randomseries2a} and \eqref{eq:randomseries2b} that $T\in \calL(X,Y)$ has Rademacher cotype $q$ if and only if it has Gaussian cotype $q$.
On the other hand, in \cite[Theorem 1C.5.3]{Montgomthesis} it is shown that for $2\leq p<q<\infty$ for all $N\geq 2$ large enough, there is a nonzero $T\in \calL(\ell^\infty_N,L^q)$ such that $C_p(T)\geq q^{-1/2} \log(N) C_p^{\gamma}(T)$.
\end{remark}

In the following result we summarize some of the known results on $R$-boundedness and $\gamma$-boundedness which will be needed.
\begin{proposition}\label{prop:wellknown}
Let $X$ and $Y$ be Banach spaces. Let $\mathscr{T}\subseteq \calL(X,Y)$.
\begin{enumerate}[(i)]
\item If $\mathscr{T}$ is $R$-bounded, then it is $\g$-bounded, and $\mathcal{R}^\g(\mathscr{T})\leq \mathcal{R}(\mathscr{T})$.
\item If $\mathscr{T}$ is $\g$-bounded then it is uniformly bounded and $\mathcal{U}(\mathscr{T}) \leq \mathcal{R}^\g(\mathscr{T})$.
\item Assume $X$ has finite cotype. If $\mathscr{T}$ is $\g$-bounded, then it is $R$-bounded, and $\mathcal{R}(\mathscr{T})\leq C \mathcal{R}^\g(\mathscr{T})$, where $C$ is a constant which only depends on $X$.
\end{enumerate}
\end{proposition}
\begin{proof}
(i) follows from the fact that $(\g_n)_{n\geq 1}$ and $(r_n \g_n)_{n\geq 1}$ have the same distribution.
(ii) is obvious. (iii) follows from \eqref{eq:randomseries2b}.
\end{proof}

\begin{remark}\label{rem:someresults} \
\begin{enumerate}[(i)]
\item For other connections between $R$-boundedness, type and cotype we refer to \cite{BlaFouSch}, \cite{vG}, \cite{HKK}, \cite{HytVer} and \cite{VerWei10}.
\item Recall the following result due to Pisier. If every uniformly bounded family is $R$-bounded then $X$ has
cotype $2$ and $Y$ has type $2$ (see \cite[Proposition 1.13]{ArBu}). The same result holds for $\g$-boundedness which follows from the same proof.
\end{enumerate}
\end{remark}

\medskip

The following lemma gives a connection between $R$-boundedness and cotype.
\begin{lemma}\label{lem:cotypeRbdd}
Let $T_1, \ldots, T_N\in\calL(\ell^\infty_M, \R)$ and let $\mathscr{T} =\{T_n: 1\leq n\leq N\}$. Let $A:\ell^\infty_M\to \ell^\infty_N$ be given by $A x = (T_n x)_{n=1}^N$. Then $\mathcal{R}(\mathscr{T}) = C_{2}(A)$ and $\mathcal{R}^\gamma(\mathscr{T}) = C_{2}^\gamma(A)$.
\end{lemma}
\begin{proof}
Let $S_1, \ldots, S_k\in \mathscr{T}$ and $x_1, \ldots, x_k\in \ell^\infty_M$. Then
\begin{align*}
\E|\sum_{i=1}^k r_i S_i x_i \Big|^2 = \sum_{i=1}^k |S_i x_i|^2 & \leq \sum_{i=1}^k \|(T_n x_i)_{n=1}^N \|_{\ell^\infty_N}^2
= \sum_{i=1}^k \|A x_i\|^2_{\ell^\infty_N} \leq C_2(A)^2 \E\Big\|\sum_{i=1}^k r_i x_i \Big\|_{\ell^\infty_M}^2
\end{align*}
and this shows that $\mathcal{R}(\mathscr{T}) \leq  C_2(A)$. Conversely, for $x_1, \ldots, x_k\in \ell^\infty_M$ choose $S_1, \ldots, S_k\in \mathscr{T}$ such that $\max_{1\leq n\leq N}|T_n x_i| = |S_i x_i|$. Then
\begin{align*}
\sum_{i=1}^k \|A x_i\|^2_{\ell^\infty_N} = \sum_{i=1}^k \|(T_n x_i)_{n=1}^N \|_{\ell^\infty_N}^2 = \sum_{i=1}^k |S_i x_i|^2\leq \mathcal{R}(\mathscr{T})^2 \E\Big\|\sum_{i=1}^k r_i x_i \Big\|_{\ell^\infty_M}^2.
\end{align*}
from which we obtain  $C_2(A)\leq \mathcal{R}(\mathscr{T})$. The proof of $\mathcal{R}^\gamma(\mathscr{T}) = C_{2}^\gamma(A)$ is similar.
\end{proof}

The next simple type of uniform boundedness principle will be used several times. For a set $S$ let $\mathcal{P}(S)$ denote its power set.
\begin{lemma}\label{lem:uniformbdd}
Let $V$ be a vector space. Let $\Phi_i:\mathcal{P}(V)\to [0, \infty]$ for $i=1, 2$ be such that
the following properties hold:
\begin{enumerate}[\rm(1)]
\item for all $A\subseteq V$ and $\lambda\in \R$, $\Phi_i(\lambda A) = |\lambda| \Phi_i(A)$.
\item If $A\subseteq B\subseteq V$, then $\Phi_i(A) \leq \Phi_i(B)$.
\item If $A_1, A_2, \cdots \subseteq V$, then  $\Phi_i\Big(\bigcup_{n=1}^\infty A_n\Big) \leq \sum_{n=1}^\infty \Phi_i(A_n)$.
\end{enumerate}
If for every $n\geq 1$ there exists a subset $B_n\subseteq V$ such that $\Phi_1(B_n)\leq 1$ and $\Phi_2(B_n)\geq c_n$ with $c_n\uparrow \infty$, then there
exists a set $A\subseteq V$ such that $\Phi_1(A) \leq 1$ and $\Phi_2(A)=\infty$.
\end{lemma}
\begin{proof}
For every $n\geq 1$ choose $A_n\subseteq V$ such that $\Phi_1(A_n)\leq 1$ and $\Phi_2(A_n)\geq 4^n$.
Setting $A = \bigcup_{n=1}^\infty 2^{-n} A_n$ one may check that the assertions hold.
\end{proof}

For $A,B\in \R$, we will write $A\lesssim_t B$ if there exists a constant $C$ depending only on $t$ such that $A \leq CB$.

\section{Proof of Theorem \ref{thm:main}}

We start with a characterization of the $R$-bound of a certain family of functionals on $c_0$.
\begin{proposition}\label{prop:Rboundedch}
Let $(a_n)_{n\geq 1}$ be scalars. Let $(T_n)_{n\geq 1}$ be the elements of $(c_0)^* = \ell^1$ given by
$T_n x = a_n x_n$.  Then $\mathcal{R}(T_n, n\geq 1) = \|a\|_{\ell^2}$.
\end{proposition}
\begin{proof}
In the sequel we write $\|\cdot\|$ for $\|\cdot\|_{c_0}$. For any $(x_n)_{n=1}^N$ one has
\begin{align*}
\Big\|\sum_{n=1}^N r_n T_n x_n\Big\|_{L^2(\O)} &= \Big(\sum_{n=1}^N |T_n x_n|^2\Big)^{1/2}
 \leq \Big(\sum_{n=1}^N \|x_n\|^2 \|T_n\|^2 \Big)^{1/2} \\ &\leq \|a\|_{\ell^2} \sup_{1\leq n\leq N}  \|x_n\|\leq  \|a\|_{\ell^2} \Big\|\sum_{n=1}^N r_n x_n\Big\|_{L^2(\O;c_0)}.
\end{align*}
By Remark \ref{rem:equal} this implies that $\mathcal{R}(T_n, n\geq 1) \leq \|a\|_{\ell^2}$.
Next choose $\eps>0$ arbitrary. Fix an integer $N\geq 1$ such that $\|a\|_{\ell^2} - \eps \leq \Big(\sum_{n=1}^N |a_n|^2\Big)^{1/2}$. Let $(x_n)_{n=1}^N$ in $c_0$ be defined by $x_{nn} = 1$ and $x_{nm} = 0$ for $m\neq n$ and $n=1, \ldots, N$. Then
\begin{align*}
\|a\|_{\ell^2} - \eps  \leq \Big(\sum_{n=1}^N a_n^2\Big)^{1/2} &= \Big(\sum_{n= 1}^N |T_n x_n|^2 \Big)^{1/2}
 =
\Big\|\sum_{n=1}^N r_n T_n x_n\Big\|_{L^2(\O)} \\ & \leq \mathcal{R}(T_n, n\geq 1) \Big\|\sum_{n\geq 1} r_n x_n\Big\|_{L^2(\O;c_0)}
\\ & = \mathcal{R}(T_n, n\geq 1)  \|\sup_{m\geq 1}\|r_m x_{mm}\|_{L^2(\O)} = \mathcal{R}(T_n, n\geq 1) .
\end{align*}
\end{proof}

In order to estimate the $\g$-bound of a specific family of coordinate functionals we need the following lemma which is a variant of \cite[Proposition 3.1, page 50]{Montgomthesis}.
Our modification of the proof  is more concise and gives a better
constant.
\begin{lemma}\label{lem:Sudakov}
Let $n\geq 1$ be fixed. Let $(x_i)_{i=1}^n$ be real numbers. Then
\begin{equation}\label{eq:estsup}
\Big(\frac{\log n}{n} \sum_{i=1}^n x_i^2\Big)^{1/2} \leq 4\,\E\sup_{i\leq n} |\gamma_i x_i|.
\end{equation}
\end{lemma}
The constant $4$ on the right-hand side of \eqref{eq:estsup} is not optimal.

\begin{proof}
It suffices to consider the case $n\geq 2$.
Without loss of generality we can assume $\E\sup_{i\leq n} |\gamma_i x_i| = 1$ and $x_i >0$ for all $i$. Fix $t>1$. Since $\P(\sup_{1\leq j\leq n} |\gamma_i x_i|>t) \leq 1/t$, it follows from \cite[Proposition 1.3.3]{KwWo} that
\[\sum_{i=1}^n \P(|\gamma_i x_i|\geq t) \leq \frac{\P(\sup_{1\leq j\leq n}  |\gamma_i x_i|>t)}{\P(\sup_{1\leq j\leq n} |\gamma_i x_i|\leq t)} \leq \frac{1}{t-1}.\]
Recalling Komatsu's bound (see \cite[Proposition 3]{SzWe}):
\[\sqrt{2\pi} \, \P(\gamma_i>s) = \int_s^\infty e^{-x^2/2} \, dx \geq \frac{2}{s+(s^2+4)^{1/2}} e^{-s^2/2},  \ \ \ s\in \R,\]
we find that with $y_i = x_i/t$
\[\frac{2}{\sqrt{2\pi}} \sum_{i=1}^n \frac{2 y_i}{1+(1+4y_i^2)^{1/2}} e^{-1/(2y_i^2)} \leq \sum_{i=1}^n \P(|\gamma_i x_i|\geq t) \leq  \frac{1}{t-1}.\]
Note that for every $i$, one has $|y_i| = t^{-1}\sqrt{\frac{\pi}{2}} \, \E |\gamma_i x_i|\leq \sqrt{\frac{\pi}{2}}$. Therefore,
\[\frac{2}{\sqrt{2\pi}}  \frac{2 y_i}{1+(1+4y_i^2)^{1/2}}\geq \frac{y_i^2}{K},\]
where $K = \frac{\pi(1+\sqrt{1+2\pi}))}{4}\approx 2.9$. Hence letting $\Theta(y) = y e^{-1/(2y)}$ we find that $\frac{1}{K} \sum_{i=1}^n \Theta(y_i^2) \leq \frac{1}{t-1}$.
Since $\Theta$ is convex we obtain that
\[\Theta\Big(\frac{1}{n}\sum_{i=1}^n y_i^2\Big) \leq \frac{K}{n(t-1)}.\]
It is straightforward to check that $\Theta(y)\geq e^{-1/y}$ for all $y>0$. Therefore, $\Theta^{-1}(u) \leq -\frac{1}{\log(u)}$ for all $u\in (0,1)$, and we obtain
\[\frac{1}{n}\sum_{i=1}^n x_i^2 \leq -\frac{t^2}{\log(K/(n(t-1)))}.\]
Now the result follows by taking $t=K+1$.
\end{proof}

\begin{remark}\label{rem:Gaussian}
A lower estimate for the constant used in \eqref{eq:estsup} follows from the following claim:
\begin{align}\label{eq:claimest}
\E \big(\sup_{i\leq n} |\gamma_i|^2\big) \leq 2\log(2n).
\end{align}
Indeed, taking $x_i = 1$ for $i=1, \ldots, n$ with $n\geq 1$ in \eqref{eq:estsup} arbitrary gives that the constant at the right-hand side of \eqref{eq:estsup} cannot be smaller than $2^{-1/2}$. To prove the claim we follow the argument in \cite[Lemma 3.2]{DGVW}. Let $\xi = \sup_{i\leq n} |\gamma_i|$ and let $h:[0,\infty)\to [1, \infty)$ be given by $h(t)= \cosh(t^{1/2})$. One easily checks that $h$ is convex and strictly increasing and $h^{-1}(s) = \log(s+(s^2-1)^{1/2})^2 \leq \log(2s)^2$. It follows from Jensen's inequality that for every $t>0$,
\begin{align*}
\E \xi^2 & =t^{-2}\E h^{-1}(\cosh(t\xi)) \leq t^{-2} h^{-1}(\E\cosh(t\xi)) \leq t^{-2} \log(2\E\cosh(t\xi))^2,
\\  \E\cosh(t\xi) & = \E \sup_{i\leq n} \cosh(t\gamma_i) \leq \sum_{i=1}^n \E \cosh(t\gamma_i) = n \E\exp(t\g_1) = n e^{t^2/2}.
\end{align*}
Combining both estimates yields that $\E\xi^2 \leq (t^{-1}\log(2n) +  t/2)^{2}$, and \eqref{eq:claimest} follows by taking
$t = \sqrt{2 \log(2n)}$.
\end{remark}

\begin{lemma}\label{lem:Gaussbound}
Let $(T_n)_{n\geq 1}$ be elements of $(c_0)^* = \ell^1$ given by
$T_n x = x_n$.  Then for all $N\geq 2$,
\[ \Big(\frac{N}{2\log 2N}\Big)^{1/2}\leq \g\big(T_n, 1\leq n\leq N\big) \leq 4 \Big(\frac{N}{\log N}\Big)^{1/2}.\]
\end{lemma}
Note that Proposition \ref{prop:Rboundedch} yields that $\mathcal{R}(T_n, 1\leq n\leq N) = N^{1/2}$, and hence there is a logarithmic improvement in the above $\g$-bound.

\begin{proof}
Fix $N\geq 2$. Let $(S_j)_{j=1}^J\subseteq \{T_n, 1\leq n\leq N\}$. We will first show that
for all $x_1, \ldots, x_J\in c_0$ one has
\begin{equation}\label{eq:gammabddnotequal}
\Big\|\sum_{j=1}^J \gamma_j S_j(x_j)\Big\|_{L^2(\O)}\leq 4 \Big(\frac{N}{\log N}\Big)^{1/2} \Big\|\sum_{j=1}^J \gamma_j x_j\Big\|_{L^2(\O;c_0)}.
\end{equation}

For $1\leq n\leq N$, let $A_n = \{j: S_j = T_n\}$. Clearly, the $(A_n)_{n=1}^N$ are pairwise disjoint.
Let $a_n = \Big(\sum_{j\in A_n} |T_n(x_j)|^2\Big)^{1/2}$ for $n=1, \ldots, N$.
It follows from orthogonality and Lemma \ref{lem:Sudakov} that
\begin{align}\label{eq:estmax}
\E_{\g}\Big|\sum_{j=1}^J \gamma_j S_j(x_j)\Big|^2 & = \sum_{j=1}^J |S_j(x_j)|^2 = \sum_{n=1}^N a_n^2 \leq \frac{16N}{\log N} \E\sup_{1\leq n\leq N} |\gamma_n a_n |^2.
\end{align}
Let $\Gamma_n= \sum_{j\in A_n} \gamma_j x_j$ for $1\leq n\leq N$. Since $(\Gamma_{nn})_{n=1}^N$ are independent Gaussian random variables and $\E|\Gamma_{nn}|^2 = a_n^2$, it follows that $(\Gamma_{nn})_{n=1}^N$ and $(\gamma_n a_n)_{n=1}^N$ have equal distributions. This yields
\begin{align}\label{eq:iddistr}
\E\sup_{1\leq n\leq N} |\gamma_n a_n |^2 = \E\sup_{1\leq n\leq N} |\Gamma_{nn}|^2.
\end{align}
For signs $(\epsilon_k)_{k\geq 1}$ let $I_{\epsilon}$ on $c_0$ be the isometry given by $I_{\epsilon} ((\alpha_k)_{k\geq 1}) =  (\epsilon_k\alpha_k)_{k\geq 1}$.
It follows that pointwise in $\O_\g$ one has
\begin{align*}
\sup_{1\leq n\leq N} |\Gamma_{nn}|^2 & = \sup_{1\leq n\leq N} \Big|\E_{r}\Big[\sum_{m=1}^N r_m r_n \Gamma_{mn}\Big]\Big|^2 \\ &
\leq \sup_{n\geq 1} \Big|\E_{r}\Big[\sum_{m=1}^N r_m r_n \Gamma_{mn}\Big]\Big|^2
= \Big\|\E_{r}\Big[ I_r \Big(\sum_{m=1}^N r_m \Gamma_m \Big)\Big]\Big\|^2
\\ &  \leq \E_{r}\Big\|I_r \Big(\sum_{m=1}^N r_m \Gamma_m \Big)\Big\|^2
=\E_{r}\Big\|\sum_{m=1}^N r_m \Gamma_m \Big\|^2,
\end{align*}
where we applied Jensen's inequality and the fact that $I_r$ is an isometry. Combining the above estimate with \eqref{eq:iddistr} and using that $\Gamma_1, \ldots, \Gamma_N$ are independent and symmetric we obtain
\begin{align*}
\E\sup_{1\leq n\leq N} |\gamma_n a_n |^2 & \leq \E_{\g}\E_{r} \Big\|\sum_{m=1}^N r_m \Gamma_m \Big\|^2
= \E_{\g} \Big\|\sum_{m=1}^N \Gamma_m\Big\|^2 = \E\Big\|\sum_{j=1}^J \gamma_j x_j\Big\|^2.
\end{align*}
Now \eqref{eq:gammabddnotequal} follows if we combine the latter estimate with \eqref{eq:estmax}.

To prove the lower estimate, let $(x_n)_{n\geq 1}$ be the standard basis for $c_0$. Let $g_N= \mathcal{R}^\g(T_n:1\leq n\leq N)$. The result follows from
\begin{align*}
N&= \E\Big|\sum_{n=1}^N \gamma_n T_n x_n\Big|^2\leq g_N^2 \E\Big\|\sum_{n=1}^N \gamma_n x_n\Big\|^2  =  g_N^2 \E \sup_{1\leq n\leq N}|\gamma_n|^2 \leq  g_N^2 2\log(2N),
\end{align*}
where we applied \eqref{eq:claimest}.
\end{proof}

As a consequence of Lemma \ref{lem:Gaussbound} we find the following result which provides an example that the Rademacher cotype and Gaussian cotype of operators are not comparable in general (cf. \cite[Theorem 1C.5.3]{Montgomthesis} and Remark \ref{rem:montgom}).
\begin{corollary}
Let $(T_n)_{n\geq 1}$ be elements of $(c_0)^* = \ell^1$ given by
$T_n x = x_n$. Let $A:\ell^\infty_N\to \ell^\infty_N$ be given by $A x = (T_n x)_{n=1}^N$. Then for all $N\geq 2$,
\[ \frac14 (\log(N))^{1/2} C_2^{\gamma}(A) \ \leq  \ C_2(A)\ \leq  \ (2\log(2N))^{1/2} C_2^{\gamma}(A).\]
\end{corollary}
\begin{proof}
This is immediate from Lemmas \ref{lem:cotypeRbdd} and \ref{lem:Gaussbound}, where we note that $C_2(A) = \mathcal{R}(\{T_n:1\leq n\leq N\} = \sqrt{N}$.
\end{proof}

We now turn to the proof of one of the main results.
\begin{proof}[Proof of Theorem \ref{thm:main}]
The implication (ii) $\Rightarrow$ (i) has already been mentioned in Proposition \ref{prop:wellknown}.

To prove (i) $\Rightarrow$ (ii) we use Lemma \ref{lem:Gaussbound}. Assume (i) holds. Assume $X$ does not have finite cotype. We will derive a contradiction. Since we may use a one-dimensional subspace of $Y$, it suffices to consider $Y = \R$. We claim that for every $N\geq 1$ there exists a $\mathscr{S}_N\subseteq \calL(X,\R)$ such that $\mathcal{R}^\g(\mathscr{S}_N)\leq 1$ and $\mathcal{R}(\mathscr{S}_N)\geq c_N$ with $c_N\uparrow \infty$ as $N\to \infty$.
For each $N\geq 1$ choose $J_N:\ell^\infty_N\to X$ and $\widehat{I}_N:X\to \ell^\infty_N$ and $X_0$ as in Corollary \ref{cor:MP1}.
Let $T_n:\ell^\infty_N\to \R$ be given by $T_n x = \frac{1}{8} \Big(\frac{\log N}{N}\Big)^{1/2} x_n$ for each $1\leq n\leq N$. Let $\mathscr{T}_N = \{T_n: 1\leq n\leq N\}$. Then as a consequence of Lemma \ref{lem:Gaussbound} we have $\mathcal{R}^\g(\mathscr{T}_N)\leq 1/2 $. From Proposition \ref{prop:Rboundedch} we find that
\[\mathcal{R}(\mathscr{T}_N)= \Big(\sum_{n=1}^N \|T_n\|^2\Big)^{1/2} = \frac{1}{8} (\log N)^{1/2}.\]
Now let $(S_n)_{n=1}^N$ be given by $S_n  = T_n \widehat{I}_N$ and $\mathscr{S}_N = \{S_n: 1\leq n\leq N\}\subseteq \calL(X,\R)$. Then by \eqref{eq:obvious} one has $\mathcal{R}^\g(\mathscr{S}_N)\leq \|\widehat{I}_N\| \mathcal{R}^\g(\mathscr{T}_N)\leq 1$. Moreover, by \eqref{eq:obvious} one has
\begin{align*}
\frac{1}{8} (\log N)^{1/2} = \mathcal{R}(\mathscr{T}_N)\leq  \mathcal{R}(\mathscr{S}_N|_{X_0}) \|J_N\|\leq \mathcal{R}(\mathscr{S}_N).
\end{align*}
Now by Lemma \ref{lem:uniformbdd} we can find a family $\mathscr{S}\subseteq \calL(X,\R)$ which is $\gamma$-bounded but not $R$-bounded. This yields a contradiction.
\end{proof}

\section{R-boundedness versus $\ell^2$-boundedness\label{sec:square}}

In this section we discuss another boundedness notion which is connected to $R$-boundedness and $\g$-boundedness.
\begin{definition}
Let $X$ and $Y$ be Banach lattices. An operator family $\mathscr{T}\subseteq \calL(X,Y)$ is called {\em $\ell^2$-bounded} if there exists a constant $C\ge 0$ such that for all $N\geq 1$, for all $(x_n)_{n=1}^N$ in $X$ and
$(T_n)_{n=1}^N$ in ${\mathscr {T}}$ we have
 \begin{equation}\label{eq:Sbdd}
 \Big\|\Big(\sum_{n=1}^N |T_n x_n|^2\Big)^{1/2} \Big\| \le C\Big\|\Big(\sum_{n=1}^N |x_n|^2\Big)^{1/2}\Big\n.
\end{equation}
The least admissible constant $C$ is called the {\em $\ell^2$-bound} of
$\mathscr {T}$. Notation $\mathcal{R}^{\ell^2}(\mathscr{T})$ or $\mathcal{R}^2(\mathscr{T})$.
\end{definition}

\begin{remark}\label{rem:latticecase}\
\begin{enumerate}[(i)]
\item The notion $\ell^2$-boundedness is the same as $R_s$-boundedness with $s=2$ as was introduced in \cite{Weis-maxreg}. A detailed treatment of the subject and applications can be found in \cite{KuUl}.
\item The square functions in \eqref{eq:Sbdd} are formed using Krivine's calculus (see \cite{LiTz}).
\item Clearly, every $\ell^2$-bounded family is uniformly bounded.
\item\label{it:singleton} A singleton $\{T\}\subseteq \calL(X,Y)$ is $\ell^2$-bounded and $\mathcal{R}^2(\{T\})\leq K_G \|T\|$,
where $K_G$ denotes the Grothendieck constant (see \cite[Theorem 1.f.14]{LiTz}).
\item\label{it:product} For lattices $X$, $Y$ and $Z$ and two families $\mathscr{T}\in \calL(X,Y)$ and $\mathscr{S}\in \calL(Y,Z)$  one has
\[
\mathcal{R}^2(\{ST: S\in \mathscr{S}, T\in \mathscr{T}\}) \leq \mathcal{R}^2(\mathscr{S}) \mathcal{R}^2(\mathscr{T}).
\]
\end{enumerate}
\end{remark}

In order to check $\ell^2$-boundedness it suffices to consider distinct operators in \eqref{eq:Sbdd}.
\begin{lemma}\label{lem:equal}
Let $X$ and $Y$ be Banach lattices and let $\mathscr{T}\subseteq \calL(X,Y)$. If there is a constant $M>0$ such that for all $N\geq 1$ and all distinct choices $T_1, \ldots, T_N\in \mathscr{T}$, one has
\[ \Big\|\Big(\sum_{n=1}^N |T_n x_n|^2\Big)^{1/2} \Big\| \le M\Big\|\Big(\sum_{n=1}^N |x_n|^2\Big)^{1/2}\Big\n, \ \ \ x_1, \ldots x_N\in X,
\]
then $\mathcal{R}^{2}(\mathscr{T})\leq K_G M$, where $K_G$ denotes the Grothendieck constant.
\end{lemma}

\begin{proof}
Let $T_1, \ldots, T_N\subseteq \mathscr{T}$ and $x_1, \ldots, x_N\in X$ be arbitrary. Let $S_1, \ldots, S_M\in \mathscr{T}$ be distinct and such that  $\{S_1, \ldots, S_M\} = \{T_1, \ldots, T_N\}$. For each $1\leq m\leq M$ let $I_m = \{i: T_i = S_m\}$. Then $(I_m)_{m=1}^M$ are disjoint sets. For each $1\leq m\leq M$ let $x_{m,i} = x_i$ if $i\in I_m$ and $x_{m,i} = 0$ otherwise.

For each $1\leq i\leq N$ let $\tilde{x}_i\in X(\ell^2_M)$ be given by $\tilde{x}_i(m) = x_{m,i}$ and let $\tilde{S}:X(\ell^2_M)\to Y(\ell^2_M)$ be given by $\tilde{S} ((y_m)_{m=1}^M) = (S_m y_m)_{m=1}^M$.
By the assumption we have that $\|\tilde{S}\|_{\calL(X(\ell^2_M),Y(\ell^2_M))}\leq \mathcal{R}^2(\mathscr{T})$. From Remark \ref{rem:latticecase} \eqref{it:singleton}, we see that
\begin{align*}
 \Big\|&\Big(\sum_{n=1}^N |T_n x_n|^2\Big)^{1/2} \Big\|_{Y}
 =  \Big\|\Big(\sum_{i=1}^N |\tilde{S} \tilde{x}_i|^2 \Big)^{1/2} \Big\|_{Y(\ell^2_M)}
 \leq K_G M \Big\|\Big(\sum_{i=1}^N |\tilde{x}_i|^2 \Big)^{1/2} \Big\|_{X(\ell^2_M)}
\\&  = K_G M \Big\|\Big(\sum_{m=1}^M \sum_{i=1}^N |x_{m,i}|^2 \Big)^{1/2} \Big\|_{X}
  = K_G M \Big\|\Big(\sum_{n=1}^N |x_{n}|^2 \Big)^{1/2} \Big\|_{X}.
\end{align*}
\end{proof}

\begin{facts}
Let $X$ be a Banach lattice and let $p\in [1, \infty)$. The following hold:
\begin{enumerate}[(i)]
\item One always has
\begin{equation}\label{eq:randomseries3a}
\Big\|\Big(\sum_{n=1}^N |x_n|^2\Big)^{1/2}\Big\|_{X} \leq \sqrt{2} \Big\|\sum_{n=1}^N r_n x_n \Big\|_{L^p(\Omega;X)},  \ \ \  x_1, \ldots, x_N\in X, \ N\geq 1.
\end{equation}
\item The space $X$ has finite cotype if and only if there is a constant $C$ such that
\begin{equation}\label{eq:randomseries3b}
\Big\|\sum_{n=1}^N r_n x_n \Big\|_{L^p(\Omega;X)}\leq C\Big\|\Big(\sum_{n=1}^N |x_n|^2\Big)^{1/2}\Big\|_{X} ,  \ \ \  x_1, \ldots, x_N\in X, \ N\geq 1.
\end{equation}
\end{enumerate}
\end{facts}
For (i) and (ii) see \cite[Theorem 16.11]{DJT} and \cite[Theorem 1.d.6]{LiTz}.

Recall that a space $X$ is {\em $2$-concave} if there is a constant $C_X$ such that for all $N\geq 1$
\[\Big(\sum_{n=1}^N \|x_n\|^2\Big)^{1/2} \leq C_X \Big\|\Big(\sum_{n=1}^N |x_n|^2\Big)^{1/2}\Big\|, \ \ x_1, \ldots, x_N\in X.\]
A space $X$ is {\em $2$-convex} if there is a constant $C_X$ such that for all $N\geq 1$
\[\Big\|\Big(\sum_{n=1}^N |x_n|^2\Big)^{1/2}\Big\|\leq C_X\Big(\sum_{n=1}^N \|x_n\|^2\Big)^{1/2}, \ \ x_1, \ldots, x_N\in X.\]

Recall the following facts from \cite[Corollary 16.9 and Theorem 16.20]{DJT} :
\begin{itemize}
\item $X$ has cotype $2$ if and only if $X$ is $2$-concave.
\item $X$ has type $2$ if and only if it has finite cotype and is 2-convex.
\end{itemize}
Note that $c_0$ is an example of a space which is $2$-convex, but does not have type $2$.

The following result is the version of Remark \ref{rem:someresults} (ii) for $\ell^2$-boundedness.
\begin{proposition}
Let $X$ and $Y$ be Banach lattices. The following are equivalent:
\begin{enumerate}[(i)]
\item Every uniformly bounded subset $\mathscr{T}\subseteq \calL(X,Y)$ is $\ell^2$-bounded.
\item $X$ is $2$-concave and $Y$ is $2$-convex.
\end{enumerate}
\end{proposition}

The proof is a slight variation of the argument in \cite{ArBu}.
\begin{proof}
(ii) $\Rightarrow$ (i):  Let $\mathscr{T}\subseteq \calL(X,Y)$ be uniformly bounded. Let $T_1, \ldots, T_N\in \mathscr{T}$ and $x_1, \ldots, x_N\in X$. If follows that
\begin{align*}
 \Big\|\Big(\sum_{n=1}^N |T_n x_n|^2\Big)^{1/2} \Big\| & \leq C_Y \Big(\sum_{n=1}^N \|T_n x_n\|^2\Big)^{1/2}
\\ & \leq C_Y \mathcal{U}(\mathscr{T}) \Big(\sum_{n=1}^N \|x_n\|^2\Big)^{1/2}
\leq C_Y \mathcal{U}(\mathscr{T}) C_X \Big\|\Big(\sum_{n=1}^N |x_n|^2\Big)^{1/2}\Big\|.
\end{align*}

(i) $\Rightarrow$ (ii): First we prove that $X$ is $2$-concave. Fix $y\in Y$ with $\|y\|=1$. Let $\mathscr{T} = \{x^*\otimes y: x^*\in X^* \ \text{with $\|x^*\|\leq 1$}\}$. Then $\mathscr{T}$ is uniformly bounded and therefore it is $\ell^2$-bounded. Choose $x_1, \ldots, x_N\in X$ arbitrary. For each $n$ choose $x_n^*\in X^*$ with $\|x_n^*\|\leq 1$ such that $\lb x_n, x_n^*\rb  = \|x_n^*\|$ and let $T_n = x^*_n\otimes y$. Then each $T_n\in \mathscr{T}$ and it follows that from \eqref{eq:randomseries3a} that
\begin{align*}
\Big(\sum_{n=1}^N \|x_n\|^2 \Big)^{1/2} & =   \Big\|\Big(\sum_{n=1}^N |T_n x_n|^2\Big)^{1/2} \Big\| \leq \mathcal{R}^2(\mathscr{T}) \Big\|\Big(\sum_{n=1}^N |x_n|^2\Big)^{1/2} \Big\|
\end{align*}

Next we show that $Y$ is $2$-convex. Fix $x\in X$ and $x^*\in X^*$ of norm one and such that $\lb x, x^*\rb=1$. Consider $\mathscr{T} = \{x^*\otimes y: y\in Y \ \text{with $\|y\|\leq 1$}\}$. Then $\mathscr{T}$ is uniformly bounded and hence $\ell^2$-bounded. Choose $y_1, \ldots, y_N\in Y$ arbitrary. Let $T_n = x^*\otimes \frac{y_n}{\|y_n\|}$ and $x_n = \|y_n\| x$ for each $n$. Then $T_1, \ldots, T_N \in \mathscr{T}$ and it follows that
\begin{align*}
\Big\|\Big(\sum_{n=1}^N |y_n|^2\Big)^{1/2} \Big\| & =   \Big\|\Big(\sum_{n=1}^N |T_n x_n|^2\Big)^{1/2} \Big\| \leq \mathcal{R}^2(\mathscr{T}) \Big\|\Big(\sum_{n=1}^N |x_n|^2\Big)^{1/2} \Big\|\leq \mathcal{R}^2(\mathscr{T}) \Big(\sum_{n=1}^N \|y_n\|^2\Big)^{1/2}.
\end{align*}
\end{proof}

\begin{theorem}\label{thm:Lmain1}
Let $X$ and $Y$ be nonzero Banach lattices. The following assertions are equivalent:
\begin{enumerate}[(i)]
\item Every $\ell^2$-bounded family $\mathscr{T}\subseteq \calL(X,Y)$ is $R$-bounded.
\item Every $\ell^2$-bounded family $\mathscr{T}\subseteq \calL(X,Y)$ is $\g$-bounded.
\item $Y$ has finite cotype.
\end{enumerate}
Moreover, in this case $\mathcal{R}(\mathscr{T})  \lesssim_Y \mathcal{R}^{2}(\mathscr{T})$ and $\mathcal{R}^\gamma(\mathscr{T})\lesssim_Y \mathcal{R}^2(\mathscr{T})$.
\end{theorem}

\begin{proof}
(i) $\Rightarrow$ (ii) follows from Proposition \ref{prop:wellknown}. To prove (iii) $\Rightarrow$ (i) assume $Y$ has finite cotype and let $\mathscr{T}$ be $\ell^2$-bounded. Fix $T_1, \ldots, T_N\in \mathscr{T}$ and $x_1, \ldots, x_N\in X$. It follows from \eqref{eq:randomseries3b} for $Y$ and \eqref{eq:randomseries3a} for $X$ that
\begin{align*}
\Big\|\sum_{n=1}^N r_n T_n x_n\Big\|_{L^2(\O;Y)} \leq C_Y \Big\|\Big(\sum_{n=1}^N |T_n x_n|^2\Big)^{1/2}\Big\|_Y
& \leq C_Y \mathcal{R}^2(\mathscr{T}) \Big\|\Big(\sum_{n=1}^N |x_n|^2\Big)^{1/2}\Big\|_X
\\ & \leq C_Y  \mathcal{R}^2(\mathscr{T}) \sqrt{2} \, \Big\|\sum_{n=1}^N r_n x_n\Big\|_{L^2(\O;X)}.
\end{align*}

To prove (ii) $\Rightarrow$ (iii) it suffices to consider $X = \R$. Assume (ii) holds and assume $Y$ does not have finite cotype.
By Corollary \ref{cor:MP1} for each $N\geq 1$ we can find $J_N:\ell^\infty_N\to Y$ and $\widehat{I}_N:Y\to \ell^\infty_N$ such that $\|\widehat{I}_N\|\leq 2$, $\|J_N\|\leq 1$ and $\widehat{I}_N J_N = \text{id}_{\ell^\infty_N}$.
Let $T_n:\R\to \ell^\infty_N$ be given by $T_n a = a e_n$. Then for $1\leq k_1, \ldots k_N\leq N$
\[\Big\|\Big(\sum_{n=1}^N |T_{k_n} a_n|^2\Big)^{1/2}\Big\|_{\ell^\infty_N} \leq \Big(\sum_{n=1}^N \|T_{k_n} a_n\|^2_{\ell^\infty_N}\Big)^{1/2} \leq \Big(\sum_{n=1}^N a_n^2\Big)^{1/2}.\]
Thus with $\mathscr{T}_N = \{T_n: \leq n\leq N\}$ we find $\mathcal{R}^2(\mathscr{T}_N)\leq 1$. On the other hand by \eqref{eq:estsup},
\begin{align*}
\tfrac14 (\log(N))^{1/2}  & \leq \Big(\E \sup_{1\leq n\leq N} |\gamma_n|^2\Big)^{1/2}
\leq \Big\|\sum_{n=1}^N \gamma_n T_n 1 \Big\|_{L^2(\O;\ell^\infty_N)}
\leq \mathcal{R}^\g(\mathscr{T}_N) \Big\|\sum_{n=1}^N \gamma_n 1\Big\|_{L^2(\O)}.
\end{align*}
This shows that $\mathcal{R}^\g(\mathscr{T}_N)\geq \tfrac14 (\log(N))^{1/2}$.
For $n\geq 1$, let $S_n:\R\to Y$ be given by $S_n = J_N T_n$ and let $\mathscr{S}_N = \{S_n:1\leq n\leq N\}$. Then by Remark \eqref{it:singleton} and \eqref{it:product}, $\mathcal{R}^2(\mathscr{S}_N)\leq K_G$. Moreover, by \eqref{eq:obvious}
\[\mathcal{R}^\g(\mathscr{S}) \geq \|\widehat{I}_N\|^{-1} \mathcal{R}^\g(T_n:1\leq n\leq N)\geq \tfrac18 (\log(N))^{1/2}.\]
Now by Lemma \ref{lem:uniformbdd} we can find a family $\mathscr{S}\subseteq \calL(\R,Y)$ which is $\ell^2$-bounded but not $\gamma$-bounded. Hence we have derived a contradiction.
\end{proof}

\begin{theorem}\label{thm:Lmain2}
Let $X$ and $Y$ be nonzero Banach lattices. The following assertions are equivalent:
\begin{enumerate}[(i)]
\item Every $R$-bounded family $\mathscr{T}\subseteq \calL(X,Y)$ is $\ell^2$-bounded.
\item Every $\gamma$-bounded family $\mathscr{T}\subseteq \calL(X,Y)$ is $\ell^2$-bounded.
\item $X$ has finite cotype.
\end{enumerate}
In this case, $\mathcal{R}^2(\mathscr{T})\lesssim_X \mathcal{R}(\mathscr{T})\eqsim_X \mathcal{R}^\gamma(\mathscr{T})$.
\end{theorem}

To prove this result we will apply some results from the theory of absolutely summing operators (see \cite{DJT}).
Let $p,q\in [1, \infty)$. An operator $T\in \calL(X,Y)$ is called {\em $(p,q)$-summing} if there is a constant $C$ such that for all $N\geq 1$ and $x_1, \ldots, x_N\in X$ one has
\[\Big(\sum_{n=1}^N \|T x_n\|^p\Big)^{1/p} \leq C \sup\Big\{\Big(\sum_{n=1}^N |\lb x_n, x^*\rb|^q\Big)^{1/q}: \|x^*\|_{X^*} \leq 1\Big\}.\]
The infimum of all $C$ as above, is denoted by $\pi_{p,q}(T)$. An operator $T\in \calL(X,Y)$ is called {\em $p$-summing} if it is $(p,p)$-summing. In this case we write $\pi_p(T) = \pi_{p,p}(T)$.

Note that in the case $X = \ell^\infty_M$ (see \cite[page 201]{DJT}),
\[\sup\Big\{\Big(\sum_{n=1}^N |\lb x_n, x^*\rb|^q\Big)^{1/q}: \|x^*\|_{X^*} \leq 1\Big\}  = \sup_{1\leq m\leq M}\Big(\sum_{n=1}^N |x_{n,m}|^q\Big)^{1/q}.\]

We provide a connection between $\ell^2$-boundedness and $2$-summing operators, which is similar as in Lemma \ref{lem:cotypeRbdd}.
\begin{lemma}\label{lem:Sandpi}
Let $T_1, \ldots, T_N\in\calL(\ell^\infty_M, \R)$ and let $\mathscr{T} =\{T_n: 1\leq n\leq N\}$. Let $A:\ell^\infty_M\to \ell^\infty_N$ be given by $A x = (T_n x)_{n=1}^N$. Then $\mathcal{R}^2(\mathscr{T}) = \pi_2(A)$.
\end{lemma}
\begin{proof}
Let $S_1, \ldots, S_k\in \mathscr{T}$ and $x_1, \ldots, x_k\in \ell^\infty_M$. Then
\begin{align*}
\sum_{i=1}^k |S_i x_i|^2 & \leq \sum_{i=1}^k \|(T_n x_i)_{n=1}^N \|_{\ell^\infty_N}^2  = \sum_{i=1}^k \|A x_i\|^2_{\ell^\infty_N} \\ & \leq \pi_2(A)^2 \sup_{1\leq m\leq M}\sum_{i=1}^k |x_{i,m}|^2 = \pi_2(A)^2 \Big\|\Big(\sum_{i=1}^k |x_i|^2\Big)^{1/2} \Big\|_{\ell^\infty_M}^2,
\end{align*}
and this shows that $\mathcal{R}^2(\mathscr{T}) \leq  \pi_2(A)$. Conversely, for $x_1, \ldots, x_k\in \ell^\infty_M$ choose $S_1, \ldots, S_k\in \mathscr{T}$ such that $\max_{1\leq n\leq N}|T_n x_i| = |S_i x_i|$. Then
\begin{align*}
\sum_{i=1}^k \|A x_i\|^2_{\ell^\infty_N} = \sum_{i=1}^k \|(T_n x_i)_{n=1}^N \|_{\ell^\infty_N}^2 = \sum_{i=1}^k |S_i x_i|^2\leq \mathcal{R}^2(\mathscr{T}) \Big\|\Big(\sum_{i=1}^k |x_i|^2\Big)^{1/2}\Big\|_{\ell^\infty_M}^2
\end{align*}
from which the result clearly follows.
\end{proof}

The next result is based on an example in \cite{Jameson} and a deep result in \cite{Talagrnew}.
\begin{lemma}\label{lem:keyPi}
Let $N\geq 3$. There exists a family $\mathscr{T} = \{T_1, \ldots, T_N\}\subseteq \calL(\ell^\infty_N,\R)$ such that
\begin{equation}\label{eq:constructTell1}
\mathcal{R}(\mathscr{T})\leq 1 \ \ \text{and} \ \ \mathcal{R}^2(\mathscr{T})\gtrsim \Big(\frac{\log(N)}{\log(\log(N))}\Big)^{1/2}.
\end{equation}
\end{lemma}
\begin{proof}
It follows from \cite[Examples 3.29 and 14.6]{Jameson} that there is an operator $A\in \calL(\ell^\infty_N)$ such that $\pi_2(A) \geq (\log(N))^{1/2}$ and $\pi_{2,1}(A)\leq 2$. Let $T_n:\ell^\infty_N\to \R$ be given by $T_n x = (A x)_n$ for $1\leq n\leq N$ and $\mathscr{T} = \{T_1,\ldots, T_N\}$. Then from Lemma \ref{lem:Sandpi} that $\mathcal{R}^2(\mathscr{T}) = \pi_2(A) \geq (\log(N))^{1/2}$. On the other hand, from Lemma \ref{lem:cotypeRbdd} and \cite[Theorem 16.1.10]{Talagrnew} we obtain
\[\mathcal{R}(\mathscr{T}) = C_2(A) \leq c (\log(\log(N)))^{1/2} \pi_{2,1}(A)\leq 2 c (\log(\log(N)))^{1/2},\]
where $c$ is a numerical constant. Now the required assertion follows by homogeneity.
\end{proof}

\begin{proof}[Proof of Theorem \ref{thm:Lmain2}]
(iii) $\Rightarrow$ (ii): Assume $X$ has finite cotype. Let $\mathscr{T}\subset\calL(X,Y)$ be $\gamma$-bounded.
Then by \eqref{eq:randomseries2a} and \eqref{eq:randomseries3a} for $Y$, and \eqref{eq:randomseries2b} and \eqref{eq:randomseries3b} for $X$, the result follows.

(ii) $\Rightarrow$ (i): Since $R$-boundedness implies $\gamma$-boundedness by Proposition \ref{prop:wellknown}, the result follows.

(i) $\Rightarrow$ (iii):
Assume that every $R$-bounded family $\mathscr{T}\subseteq \calL(X,\R)$ is $\ell^2$-bounded. Assuming that $X$ does not have finite cotype, one can use the same construction as in Theorem \ref{thm:main} but this time applying Lemma \ref{lem:keyPi} instead of Lemma \ref{lem:Gaussbound}. Here one also needs to apply Remark \ref{rem:latticecase} in a similar way as in Theorem \ref{thm:Lmain1}.
\end{proof}

\section{Duality and $R$-boundedness\label{sec:dual}}

In this final section we consider duality of $R$-boundedness, $\gamma$-boundedness and $\ell^2$-boundedness.
For a family $\mathscr{T}\subseteq \calL(X,Y)$ we write $\mathscr{T}^* = \{T^*:T\in \calL(X,Y)\}$.

For $\ell^2$-boundedness, there is a duality result which does not depend on the geometry of the spaces.
\begin{proposition}\label{prop:easyudual}
Let $X$ and $Y$ be Banach lattices. A family $\mathscr{T}\subseteq\calL(X,Y)$ is $\ell^2$-bounded if and only if $\mathscr{T}^*$ is $\ell^2$-bounded.
In this case $\mathcal{R}^2(\mathscr{T}) = \mathcal{R}^2(\mathscr{T}^*)$.
\end{proposition}
\begin{proof}
This easily follows from the fact that for Banach lattices $E$, one has $E(\ell^2_N)^* = E^*(\ell^2_N)$ isometrically (see \cite[p.\ 47]{LiTz}).
\end{proof}

Recall from \cite{dePaRick} and \cite{HKK} that a family $\mathscr{T}\subseteq \calL(X,Y)$ is $R$-bounded if and only if $\mathscr{T}^{**}\subseteq \calL(X^{**},Y^{**})$ is $R$-bounded. The same holds for $\gamma$-boundedness.
It is well-known that for spaces with nontrivial type (or equivalently $K$-convex with respect to the Rademacher system by Pisier's theorem, see \cite[Chapter 13]{DJT}), $R$-boundedness of $\mathscr{T}\subseteq \calL(X,Y)$ implies $R$-boundedness of $\mathscr{T}^*\subseteq \calL(Y^*,X^*)$ (see \cite[Lemma 3.1]{KWcalc}). By \cite[Corollary 2.8]{PisConv} the same method can be used to obtain duality for $\gamma$-boundedness.
The following result shows that the geometric limitation of nontrivial type is also necessary:

\begin{theorem}\label{thm:duality}
Let $X$ and $Y$ be Banach spaces. The following are equivalent:
\begin{enumerate}[(i)]
\item For every $R$-bounded family $\mathscr{T}\subseteq \calL(X,Y)$, the family $\mathscr{T}^*\subseteq \calL(Y^*,X^*)$ is $R$-bounded.
\item For every $R$-bounded family $\mathscr{T}^*\subseteq \calL(X^*,Y^*)$, the family $\mathscr{T}\subseteq \calL(Y,X)$ is $R$-bounded.
\item For every $\gamma$-bounded family $\mathscr{T}\subseteq \calL(X,Y)$, the family $\mathscr{T}^*\subseteq \calL(Y^*,X^*)$ is $\gamma$-bounded.
\item For every $\gamma$-bounded family $\mathscr{T}^*\subseteq \calL(X^*,Y^*)$, the family $\mathscr{T}\subseteq \calL(Y,X)$ is $\gamma$-bounded.
\item $X$ has nontrivial type.
\end{enumerate}
In this case for every $\mathscr{T}\subseteq \calL(X,Y)$,
\[\mathcal{R}^{\gamma}(\mathscr{T}) \eqsim_X \mathcal{R}(\mathscr{T}) \eqsim_X \mathcal{R}(\mathscr{T}^*) \eqsim_X \mathcal{R}^{\gamma}(\mathscr{T}^*).\]
\end{theorem}

\begin{proof}
(v) $\Rightarrow$ (i) and (v) $\Rightarrow$ (iii): See the references before Theorem \ref{thm:duality}.

(i) $\Rightarrow$ (v): Assume (i) and assume $X$ does not have nontrivial type. From Theorem \ref{thm:MP2} it follows that for every $N\geq 1$, there exists $J_N:\ell^1_N\to X^*$ such that $\tfrac12\|z \|\leq \|J_N z\|\leq \|z\|$.
Let $\mathscr{T}_N\subseteq \calL(\ell^\infty_N,\R)$ be as in \eqref{eq:constructTell1}. Then $\mathcal{R}(\mathscr{T}_N)\leq 1$. Moreover,
since $\R$ has cotype $2$ it follows from Theorem \ref{thm:Lmain2}, Proposition \ref{prop:easyudual} and \eqref{eq:constructTell1} that
\begin{equation}\label{eq:RbounddualcN}
\mathcal{R}(\mathscr{T}_N^*) \gtrsim \mathcal{R}^2(\mathscr{T}_N^*) = \mathcal{R}^2(\mathscr{T}_N)\gtrsim \Big(\frac{\log(N)}{\log(\log(N))}\Big)^{1/2}=:c_N,
\end{equation}
Therefore, there is a constant $K$ such that $\mathcal{R}(\mathscr{T}_N^*)\geq K c_N$.

Now let $\mathscr{S}_N = \{TJ_N^*|_{X}: T\in \mathscr{T}_N\}\subseteq \calL(X,\R)$. Then $\mathcal{R}(\mathscr{S}_N)\leq 1$. Furthermore, noting that $(J_N^*|_{X})^* = J_N$ and hence $J_N T^*\in \mathscr{S}_N^*$ for all $T\in \mathscr{T}_N$, one obtains
\[ K c_N \leq \mathcal{R}(\mathscr{T}_N^*)  \leq 2 \mathcal{R}(\mathscr{S}_N^*).\]

Therefore, $\mathcal{R}(\mathscr{S}_N^*)\geq \tfrac12 \mathcal{R}(\mathscr{T}_N^*) \geq \tfrac{K}{2} c_N$.
Now by Lemma \ref{lem:uniformbdd} we can find a family $\mathscr{S}\subseteq \calL(X,\R)$ which is $\ell^2$-bounded but not $R$-bounded.

(iii) $\Rightarrow$ (v): This follows from the proof of (i) $\Rightarrow$ (v). Indeed, for the example in (i) $\Rightarrow$ (v) one has $\mathscr{S}$ is $R$-bounded and hence $\gamma$-bounded by Proposition \ref{prop:wellknown}. Since, $\mathscr{S}^*$ is not $R$-bounded, Proposition \ref{prop:wellknown} and the finite cotype of $\R$ imply that $\mathscr{S}^*$ is also not $\gamma$-bounded.

(ii) $\Rightarrow$ (v) and (iv) $\Rightarrow$ (v): These can be proved in a similar way as (i) $\Rightarrow$ (v) and (iii) $\Rightarrow$ (v) respectively. This time use $J_N:\ell^1_N\to X$ such that $\tfrac12\|z \|\leq \|J_N z\|\leq \|z\|$ and let $\mathscr{S}_N = \{J_N T^*: T\in \mathscr{T}_N\}\subseteq \calL(\R, X)$. Then $\mathcal{R}(\mathscr{S}_N)$ is unbounded in $N$ and $\mathcal{R}(\mathscr{S}_N^*)\leq 1$. Here $\mathscr{S}_N^* = \{TJ_N^*: T\in \mathscr{T}_N\}\subseteq \calL(X^*,\R)$.

(v) $\Rightarrow$ (ii) and (v) $\Rightarrow$ (iv): If $X$ has nontrivial type, then $X^*$ has nontrivial type. Therefore, the results follow from (v) $\Rightarrow$ (i) and (v) $\Rightarrow$ (ii) applied to $X^*$.
\end{proof}

\def\polhk#1{\setbox0=\hbox{#1}{\ooalign{\hidewidth
  \lower1.5ex\hbox{`}\hidewidth\crcr\unhbox0}}} \def\cprime{$'$}

\end{document}